\documentclass[12pt]{elsarticle}

\usepackage{amssymb}
\usepackage{amsmath}
\usepackage{commath}

\usepackage{xcolor}
\usepackage{hyperref}
\hypersetup{breaklinks=true, hidelinks}

\newtheorem{theorem}{Theorem}[section]
\newtheorem{lemma}[theorem]{Lemma}
\newproof{proof}{Proof}
\newdefinition{definition}{Definition}
\newdefinition{remark}{Remark}

\begin{document}

\begin{frontmatter}

\title{Sign changes in Fourier coefficients of the symmetric power $L$-functions on sums of two squares}

\author{Amrinder Kaur} 

\affiliation{organization={Harish-Chandra Research Institute},
            addressline={Chhatnag Road, Jhunsi}, 
            city={Prayagraj},
            postcode={211019}, 
            state={Uttar Pradesh},
            country={India}}

\begin{abstract}
Let $f$ be a normalized primitive Hecke eigen cusp form of even integral weight $k$ for the full modular group $SL(2,\mathbb{Z})$. For integers $j \geq 2$, let $\lambda_{sym^j f}(m)$ denote the $m$th Fourier coefficient of the $j$th symmetric power $L$-function associated with $f$. We give a quantitative result on the number of sign changes of $\lambda_{sym^j f}(m)$ for the indices $m$ that are the sum of two squares in the interval $[1,x]$ for sufficiently large $x$.
\end{abstract}


\end{frontmatter}

\section{Introduction}

The study of sign changes in Fourier coefficients has long been of interest to number theorists. Given a function $h: \mathbb{N} \to \mathbb{R}$, we say that it has $l$ sign changes in $[1,x]$ if there exist integers $1 \leq m_1 \leq m_2 \leq \cdots \leq m_{l+1} \leq x$ such that $h(m_i) \neq 0$ for all $1 \leq i \leq l+1$ and $h(m_i)h(m_{i+1})<0$ for all $1\leq i \leq l$. Matom\"aki and Radziwi\l\l \cite[Theorem 1.2]{KmMr1} showed that there is a positive proportion of sign changes in the Fourier coefficients of modular forms and Maass forms for $SL(2,\mathbb{Z})$. 
In a subsequent work \cite[Corollary 3]{KmMr2}, they established a broader result: any real multiplicative function has a positive proportion of sign changes if and only if it attains a negative value at some integer and is non-zero for a positive proportion of integers. This implies that the Fourier coefficients of the symmetric power $L$-function associated to modular forms exhibit a positive proportion of sign changes. Some significant results regarding the sign changes of Fourier coefficients of higher degree $L$-functions have been established by Meher and Murty \cite[Theorem 6.1]{JmMrm2}, J{\"a}{\"a}saari \cite[Theorem 1]{Jj} and Kim \cite[Corollary 1.11]{Jk}.

 The study of sign changes in subsequences of Fourier coefficients has also been explored in various contexts. In the case of sparse subsequences, Matom\"aki and Radziwi\l\l's result \cite[Corollary 3]{KmMr2} is inapplicable, though one still expects that the number of sign changes is of the same order of magnitude as the number of nonzero values. Banerjee and Pandey \cite[Theorem 2.1]{SbMkp} studied the sign changes of Fourier coefficients of modular forms on the sum of two squares. Lowry-Duda later improved their result in his preprint \cite[Theorem 1]{Dld}. The author, in collaboration with Saha \cite{AkBs}, examined the sign changes of Fourier coefficients of $SL(2,\mathbb{Z})$ Maass forms on the sum of two squares. Motivated by these works, we consider studying the sign changes of Fourier coefficients of higher degree $L$-functions in a subsequence. 

Let $f$ be a normalized primitive Hecke eigen cusp form of even integral weight $k$ for the full modular group $SL(2,\mathbb{Z})$ and $H_k$ be the set of all such forms. Then $f$ admits the Fourier expansion at the cusp $\infty$:
$$ f(z) = \sum_{m=1}^{\infty} \lambda_f(m) m^{\frac{k-1}{2}} e^{2 \pi i mz} ,$$
where $\lambda_f(m)$ is the $m$th Fourier coefficient and the Hecke eigenvalue. This Fourier coefficient is normalized so that $\lambda_f(1)=1$. Let $\lambda_{sym^j f}(m)$ denote the $m$th Fourier coefficient of the $j$th symmetric power $L$-function associated with $f$.

Utilizing the axiomatization introduced by Meher and Murty \cite[Theorem 1.1]{JmMrm1}, we aim to study the sign changes for the Fourier coefficients of the symmetric power $L$-function $L_{\text{sym}^j f}(s)$ on the sum of two squares. More precisely, we analyze the sign changes in the sequence
$$\{ \lambda_{\text{sym}^j f}(m) \mid m \leq x , m = c^2+d^2, (c,d) \in \mathbb{Z}^2 \}$$ 
for sufficiently large $x>0$. In particular, we show the following result.

\begin{theorem} \label{t1}
Let $f \in H_k$. For any integer $j \geq 2$ and sufficiently large $x$, the sequence
$$ \{ \lambda_{\text{sym}^j f}(m) \mid m \leq x , m = c^2+d^2, (c,d) \in \mathbb{Z}^2 \} $$
has at least $x^{1-\delta_j}$ sign changes for any $\delta_j$ with 
$$ \frac{21j^2+42j+19}{21j^2+42j+40} < \delta_j <1. $$
\end{theorem}

\begin{remark}
\normalfont{
The result mentioned in Theorem \ref{t1} holds for the Fourier coefficients of the $j$th symmetric power $L$-function associated with Maass forms on $SL(2,\mathbb{Z})$ under the assumption of Ramanujan conjecture and the automorphy of $L_{sym^i f}(s)$ for $1 \leq i \leq 2j$.
}
\end{remark}

\section{Background and Notations}

We begin by recalling some key definitions in this section which are useful in understanding the behaviour of Fourier coefficients. 

\begin{definition} \cite[Section 14.5]{HiEk}
The Hecke $L$-function attached to $f$ is defined as
$$ L_f(s) = \sum_{m=1}^{\infty} \frac{\lambda_f(m)}{m^s}. $$
It is absolutely convergent on $\Re(s) > 1$ and in this region has the Euler product
$$ L_f(s) = \prod_p \left( 1-\frac{\alpha_f(p)}{p^s} \right)^{-1} \left( 1-\frac{\beta_f(p)}{p^s} \right)^{-1} .$$
The complex coefficients $\alpha_f(p)$ and $\beta_f(p)$ satisfy the conditions
$$ \alpha_f(p)+\beta_f(p) = \lambda_f(p) \ \text{and} \ \alpha_f(p) \beta_f(p) = \abs{\alpha_f(p)} = \abs{\beta_f(p)} = 1. $$
\end{definition}

Let 
$$ r_2(m) := \# \{ m= c^2+d^2 \mid (c,d) \in \mathbb{Z}^2 \} $$
denote the number of ways of writing $m$ as the sum of two squares. Then $r_2(m)$ is non-negative. 
The classical theta function is defined as 
$$ \theta(z) = \sum_{m \in \mathbb{Z}} e^{2 \pi i m^2 z} $$
which is a modular form of weight $\frac{1}{2}$ on $\Gamma_0(4)$. Then
$$ \theta^2(z) = 1+\sum_{m=1}^{\infty} r_2(m) e^{2 \pi i m z} $$
is a modular form of weight 1 on $\Gamma_0(4)$ with character $\chi$ being the primitive Dirichlet character modulo 4. Precisely,
$$ \chi(d) =
\begin{cases}
1 & d \equiv 1 \ (\text{mod} \ 4) \\
-1 & d \equiv 3 \ (\text{mod} \ 4) \\
0 & d \equiv 0,2 \ (\text{mod} \ 4) 
\end{cases}
$$
and
$$ r_2(m) := 4 r(m) \ \text{where} \ r(m):= \sum_{d|m} \chi(d). $$

\begin{definition} \cite[Section 3.2]{HiEk}
The Dirichlet $L$-function associated to $\chi$ is defined by
$$ L_{\chi}(s) = \sum_{m=1}^{\infty} \frac{\chi(m)}{m^s} $$
which is absolutely convergent on $\Re(s)>1$.
\end{definition}

\begin{definition} \cite[Section 13.8]{Hi}
For integers $j \geq 0$, the $j$-th symmetric power $L$-function attached to $f$ is defined as
\begin{align*}
L_{sym^j f}(s) &= \prod_p \prod_{m=0}^j \left( 1 - \frac{\alpha_f(p)^{j-m} \beta_f(p)^m}{p^s} \right)^{-1} \\
&:= \sum_{m=1}^{\infty} \frac{\lambda_{sym^j f}(m)}{m^s}
\end{align*}
and it is absolutely convergent on $\Re(s)>1$.
\end{definition}

Note that 
$$ L_{sym^0 f}(s) = \zeta(s) \ \text{and} \ L_{sym^1 f}(s) = L_f(s). $$

\begin{definition}
The twisted $j$-th symmetric power $L$-function can be defined as
$$ L_{sym^j f \otimes \chi}(s) = \sum_{m=1}^{\infty} \frac{\lambda_{sym^j f}(m) \chi(m)}{m^s} $$
which is absolutely convergent for $\Re(s)>1$ and has the Euler product
$$ L_{sym^j f \otimes \chi}(s) = \prod_p \prod_{m=0}^j \left( 1- \frac{\alpha_f(p)^{j-m} \beta_f(p)^m \chi(p)}{p^s} \right)^{-1}. $$
\end{definition}

\section{Auxiliary results}

In this section, we reproduce some subconvexity results and investigate the asymptotics for partial sums $\sum_{m \leq x} \lambda_{sym^j f}(m) r_2(m)$ and $\sum_{m \leq x} \lambda_{sym^j f}^2(m) r_2(m)$.

\begin{lemma}{\cite[Theorem 5]{Jb}} \label{l1}
For $\epsilon>0$, we have
$$ \zeta(\sigma+it) \ll (\, \abs{t}+1)^{\frac{13}{42}(1-\sigma)+\epsilon} $$ 
uniformly for $\frac{1}{2} \leq \sigma \leq 1$ and $\, \abs{t} \geq 1$.
\end{lemma}

\begin{lemma} \label{l2}
For any $\epsilon>0$, $\frac{1}{2} \leq \sigma \leq 1$ and $\, \abs{t} \geq 1$, we have
$$ L_{sym^2 f}(\sigma+it) \ll (\, \abs{t}+1)^{\frac{8}{7}(1-\sigma)+\epsilon}. $$
\end{lemma}

\begin{proof}
By Gelbart--Jacquet lift \cite[Proposition 3.2]{SgHj}, the symmetric square of $f$ corresponds to a self-dual $SL(3,\mathbb{Z})$ Hecke--Maass form (see \cite[Definition 5.1.3]{Dg} for the definition of $SL(n,\mathbb{Z})$ Hecke--Maass form for $n \geq 2$ and \cite[Proposition 9.2.1]{Dg} for the definition of dual Hecke--Maass form). For self-dual $SL(3,\mathbb{Z})$ forms, the above subconvexity bound can be found in the preprint of Dasgupta, Leung and Young \cite[Corollary 1.4]{AdWhlMpy}.
\end{proof}

\begin{lemma} \label{l3}
For $j \geq 3$ and $\epsilon>0$, we have
$$ L_{sym^j f}(\sigma+it) \ll (\, \abs{t}+1)^{\frac{j+1}{2}+\epsilon} $$
which holds uniformly in $\frac{1}{2} \leq \sigma \leq 1$ and $\, \abs{t} \geq 1$.
\end{lemma}

\begin{proof}
In the framework of Perelli \cite{Ap}, $L_{sym^j f}(s)$ is a general $L$-function, and this bound holds for such general $L$-functions.
\end{proof}

\begin{lemma} \label{l4}
For $f \in H_k$, consider $L_{sym^j f}(s)$ for integers $j \geq 0$. Let $\chi$ be the non-trivial Dirichlet character modulo 4. For any $\epsilon>0$, $\frac{1}{2} \leq \sigma \leq 1$, $\, \abs{t} \geq 1$, if we have
$$ L_{sym^j f}(\sigma+it) \ll (\, \abs{t}+1)^{\alpha(1-\sigma)+\epsilon} $$
for some $\alpha \in \mathbb{R}^+$, then we have
$$ L_{sym^j f \otimes \chi}(\sigma+it) \ll (\, \abs{t}+1)^{\alpha(1-\sigma)+\epsilon}. $$
\end{lemma} 

\begin{proof}
The breakthrough work of Newton and Thorne \cite{JnJat1, JnJat2} gives the automorphy of $L_{sym^j f}(s)$ for all $j \geq 1$. This implies that for $j \geq 1$, $L_{sym^j f}(s)$ can be analytically continued as an entire function and satisfies a functional equation of the Riemann zeta type. Hence, for $j \geq 0$ and $\chi$ a non-trivial Dirichlet character modulo 4, the twisted $L$-function $L_{sym^j f \otimes \chi}(s)$ has an analytic continuation to the entire complex plane. 

For a primitive character $\psi$, $L_{sym^j f \otimes \psi}(s)$ admits a functional equation relating $L_{sym^j f \otimes \psi}(s)$ to $L_{sym^j f \otimes \overline{\psi}}(1-s)$, where $\overline{\psi}$ is the complex conjugate of $\psi$. Since in our case, $\chi$ is a primitive real character, we get a functional equation relating $L_{sym^j f \otimes \chi}(s)$ to $L_{sym^j f \otimes \chi}(1-s)$.
 
If an $L$-function has an Euler product of degree $n \geq 1$ and functional equation of the Riemann zeta type, then it satisfies an approximate functional equation whose general form can be seen in \cite[Theorem 5.3]{HiEk}. For proving a subconvexity result, we start with an approximate functional equation whose main term is contributed by partial sums of $L_{sym^j f \otimes \chi}(\frac{1}{2}+it)$ with $\, \abs{t} \geq 1$. These partial sums can be bounded by the partial sum of $L_{sym^j f}(\frac{1}{2}+it)$. Using the subconvexity estimate for $L_{sym^j f}(\frac{1}{2}+it)$ along with the Phragm\'en--Lindel\"of principle, we obtain the same subconvexity bound in the $t$-aspect for $L_{sym^j f \otimes \chi}(\sigma+it)$ with $\sigma$ lying in the critical strip.
\end{proof}

Utilizing the above results, we get the following subconvexity bounds for the twisted $L$-functions.

\begin{lemma} \label{l5}
For integers $j \geq 3$ and $\epsilon>0$, we have the bounds
\begin{align*}
L_{\chi}(\sigma+it) &\ll (\, \abs{t}+1)^{\frac{13}{42}(1-\sigma)+\epsilon} ,\\
L_{sym^2 f \otimes \chi}(\sigma+it) &\ll (\, \abs{t}+1)^{\frac{8}{7}(1-\sigma)+\epsilon} ,\\
L_{sym^j f \otimes \chi}(\sigma+it) &\ll (\, \abs{t}+1)^{\frac{j+1}{2}+\epsilon} ,
\end{align*}
which hold uniformly in $\frac{1}{2} \leq \sigma \leq 1$ and $\, \abs{t} \geq 1$.
\end{lemma}

\begin{lemma} \label{l6}
Let $\epsilon>0$ and $m \geq 1, j \geq 0$ be integers. Then we have the bound
$$ \lambda_{sym^j f}(m) r_2(m) \ll m^{\epsilon}. $$
\end{lemma}

\begin{lemma} \label{l7}
Let $j \geq 2$ be an integer. For $\epsilon>0$ and sufficiently large $x$, we get
$$ \sum_{m \leq x} \lambda_{sym^j f}(m) r_2(m) \ll x^{\frac{j}{j+2}+\epsilon} .$$
\end{lemma}

\begin{proof}
This result is a direct consequence of Tang and Wu \cite[Theorem 1]{HtJw}.
\end{proof}

In order to study the partial sum $\sum_{m \leq x} \lambda_{sym^j f}^2(m) r_2(m)$, we consider the Dirichlet series
$$ F(s) := \sum_{m=1}^{\infty} \frac{\lambda_{sym^j f}^2(m) r(m)}{m^s} ,$$
which is absolutely convergent on $\Re(s)>1$.

\begin{lemma} \label{l8}
When $\Re(s)>1$, we have the decomposition $F(s) = G(s) H(s)$ where
$$ G(s) := \prod_{i=0}^{j} L_{sym^{2i} f}(s) L_{sym^{2i} f \otimes \chi}(s)  $$
and $H(s)$ is absolutely convergent on $\Re(s) \geq \frac{1}{2}+\epsilon$ with $H(1) \neq 0$. 
\end{lemma}

\begin{proof}
The Dirichlet series $F(s)$ has the following Euler product
$$ F(s) = \prod_p \left( 1+ \frac{\lambda_{sym^j f}^2(p)r(p)}{p^s} + \frac{\lambda_{sym^j f}^2(p^2) r(p^2)}{p^{2s}} + \cdots \right). $$

The coefficient of $p^{-s}$ can be written as
\begin{align*}
\lambda_{sym^j f}^2(p) r(p) &= \left( \sum_{i=0}^j \lambda_{sym^{2i} f}(p) \right) \left(1+\chi(p) \right) \\
&= \sum_{i=0}^j \left( \lambda_{sym^{2i} f}(p) + \lambda_{sym^{2i} f}(p) \chi(p) \right).
\end{align*}

This gives us that $F(s) = G(s) H(s)$ where 
$$ G(s) = \prod_{i=0}^{j} L_{sym^{2i} f}(s) L_{sym^{2i} f \otimes \chi}(s). $$
We can write the Euler product of $G(s)$ as
$$ G(s) = \prod_p \left( 1+ \frac{g(p)}{p^s} + \frac{g(p^2)}{p^{2s}} + \cdots \right) $$
where 
$$ g(p) = \sum_{i=0}^j \left( \lambda_{sym^{2i} f}(p) + \lambda_{sym^{2i} f}(p) \chi(p) \right) \ll p^{\epsilon}.$$ 

The Dirichlet series $H(s)$ will have the Euler product
$$ H(s) = \prod_p \left( 1+\frac{h(p^2)}{p^{2s}} + \frac{h(p^3)}{p^{3s}} + \cdots \right) $$
where 
$$ h(p^2) = \lambda_{sym^j f}^2(p^2) r(p^2) - g(p^2) \ll p^{\epsilon}. $$

The infinite product $H(s)$ converges absolutely when
$$ \sum_{p} \frac{h(p^2)}{p^{2s}} $$
converges absolutely i.e., when $\Re(s) \geq \frac{1}{2}+\epsilon$. Convergence of the product $H(s)$ implies that $H(1) \neq 0$.
\end{proof}

\begin{lemma} \label{l9}
For an integer $j \geq 2$ and sufficiently large $x$, we have
$$ \sum_{m \leq x} \lambda_{sym^j f}^2(m) r_2(m) = C_jx + O(x^{\gamma_j}) $$
where $C_j$ is a constant that depends on the form $f$ and $j$ and
$$ \gamma_j = \frac{21j^2+42j+19}{21j^2+42j+40}. $$
\end{lemma}

\begin{proof}
By Perron's formula \cite[Exercise 4.4.15]{Mrmbook},
\begin{align*}
\sum_{m \leq x} \lambda_{sym^j f}^2(m) r_2(m) &= 4 \sum_{m \leq x} \lambda_{sym^j f}^2(m) r(m) \\
&= \frac{4}{2 \pi i} \int_{1+\epsilon-iT}^{1+\epsilon+iT} \frac{F(s)x^s}{s} ds + O \left(\frac{x^{1+\epsilon}}{T} \right).
\end{align*}

By Lemma \ref{l8}, $F(s) = G(s) H(s)$ where $G(s)$ has a simple pole at $s=1$ and $H(s)$ is absolutely convergent on $\Re(s) \geq  \frac{1}{2}+\epsilon$. By moving the line of integration to $\Re(s) = \frac{1}{2}+\epsilon$, we can collect the residue of $F(s)$ at $s=1$, which contributes the main term. This is
\begin{align*}
4 \, \text{Res}_{s=1} \frac{F(s)x^s}{s} &= 4 \lim_{s \to 1} \frac{(s-1) \prod_{i=0}^{j} \left( L_{sym^{2i} f}(s) L_{sym^{2i} f \otimes \chi}(s) \right) H(s) x^s}{s} \\
&= 4 L_{\chi}(1) \prod_{i=1}^{j} \left( L_{sym^{2i} f}(1) L_{sym^{2i} f \otimes \chi}(1) \right) H(1) x = C_j x.
\end{align*}

Therefore, we get
\begin{align*}
\sum_{m \leq x} \lambda_{sym^j f}^2(m) r_2(m) &= C_j x + \frac{2}{\pi i} \left[ \int_{1+\epsilon-iT}^{\frac{1}{2}+\epsilon -iT} + \int_{\frac{1}{2}+\epsilon-iT}^{\frac{1}{2}+\epsilon+iT} + \int_{\frac{1}{2}+\epsilon+iT}^{1+\epsilon+iT} \right] \frac{F(s) x^s}{s} ds \\
&\quad+ O \left( \frac{x^{1+\epsilon}}{T} \right). 
\end{align*}

Substituting the bounds from Lemmas \ref{l1}, \ref{l2}, \ref{l3} and \ref{l5} in the decomposition from Lemma \ref{l8}, we obtain the bound
$$ F(\sigma+it) \ll (\, \abs{t}+1)^{\left( j^2+2j-\frac{2}{21} \right)(1-\sigma)+\epsilon}$$
which holds uniformly in $\frac{1}{2} \leq \sigma \leq 1$ and $\, \abs{t} \geq 1$.

It is easy to see that the contribution from horizontal lines is bounded by
$$ \frac{x^{1+\epsilon}}{T} + x^{\frac{1}{2}+\epsilon} T^{\frac{j^2}{2}+j - \frac{22}{21}}. $$

The vertical line integral is split into two integrals over $\, \abs{t} \leq 1$ and $1<\, \abs{t} \leq T$. Over the integral $\, \abs{t} \leq 1$, the integral is bounded by $x^{\frac{1}{2}+\epsilon}$ and over $1< \, \abs{t} \leq T$, the integral is bounded by 
$$x^{\frac{1}{2}+\epsilon} T^{\frac{j^2}{2}+j-\frac{1}{21}}.$$

Combining all the error terms, we obtain
$$ \sum_{m \leq x} \lambda_{sym^j f}^2(m) r_2(m) = C_j x + O\left( x^{\frac{1}{2}+\epsilon} T^{\frac{j^2}{2}+j-\frac{1}{21}} \right) + O \left( \frac{x^{1+\epsilon}}{T} \right). $$
We choose $T$ such that $ x^{\frac{1}{2}+\epsilon} T^{\frac{j^2}{2}+j-\frac{1}{21}} \sim \frac{x^{1+\epsilon}}{T} $ i.e., $T \sim x^{\frac{21}{21j^2+42j+40}}$. Substituting this value of $T$, we get our result. 
\end{proof}

\section{Proof of Theorem \ref{t1}}

We will apply a weighted version of the axiomatization for sign changes provided by Meher and Murty \cite[Theorem 1.1]{JmMrm1}. Their result tells us that if we have
\begin{align*}
\lambda_{sym^j f}(m) r_2(m) &\ll m^{\alpha_j}, \\
\sum_{m \leq x} \lambda_{sym^j f}(m) r_2(m) &\ll x^{\beta_j}, \\
\sum_{m \leq x} \lambda_{sym^j f}^2(m) r_2(m) &= Cx + O(x^{\gamma_j}), 
\end{align*}
then there are at least $x^{1-\delta_j}$ sign changes of the sequence $ \{ \lambda_{sym^j f}(m) r_2(m) \mid m \leq x \} $ where
$$ \max \{ \alpha_j + \beta_j , \gamma_j \} < \delta_j <1.$$

From the lemmas \ref{l6}, \ref{l7} and \ref{l9}, we obtain
\begin{align*}
\alpha_j &= \epsilon ,\\
\beta_j &= \frac{j}{j+2} + \epsilon ,\\
\gamma_j &= \frac{21j^2+42j+19}{21j^2+42j+40},
\end{align*}
which completes the proof of Theorem \ref{t1}.

\section*{Acknowledgements}
The author is thankful to the Department of Atomic Energy, Government of India, for providing financial support for this work and the Harish-Chandra Research Institute, a CI of Homi Bhabha National Institute, for providing the necessary research facilities.

\bibliographystyle{elsarticle-num} 
\bibliography{Ref}

\end{document}